\documentclass[a4paper,10pt]{article}
\usepackage{latexsym}
\usepackage[ansinew]{inputenc}

\usepackage{amsthm}
\usepackage{amsmath}
\usepackage{amssymb}
\usepackage{amsfonts}
\usepackage{color}

\usepackage{array}
\usepackage{float}
\usepackage{caption}
\usepackage{graphicx}
\usepackage{anysize}

\usepackage{tabularx}

\setcaptionwidth{0.95\textwidth}%

\newcommand{\beq}{\begin{equation}}
\newcommand{\enq}{\end{equation}}
\newtheorem{lemma}{Lemma}
\newtheorem{proposition}{Proposition}

\newtheorem{corollary}{Corollary}
\newtheorem{definition}{Definition}
\newtheorem{remark}{Remark}

\newcommand{\cp}{\mathcal{C}}
\newcommand{\acp}{{\rm{Aut}}(\cp)}

\newcommand{\DdE}{{\mathcal{L}^2}(E)}
\newcommand{\Dv}{{\mathcal{L}^1}(v)}
\newcommand{\Ddv}{{\mathcal{L}^2}(v)}
\newcommand{\Dw}{{\mathcal{L}^1}(w)}
\newcommand{\Ddw}{{\mathcal{L}^2}(w)}

\normalsize 

\title{On the automorphism group of the asymptotic pants complex
of a planar surface of infinite type}
\author{Ariadna Fossas $\qquad$ Maxime Nguyen\\
Institut Fourier\\
Universit\'e de Grenoble I\\
38402 St. Martin d'H\`eres. France\\
ariadna.fossas@ujf-grenoble.fr\\
maxime.nguyen@ujf-grenoble.fr}
\date{}

\begin{document}

\maketitle

\begin{abstract}
We consider a planar surface $\Sigma$ of infinite type which has 
the Thompson group $\mathcal{T}$ as asymptotic mapping class group.
We construct the asymptotic pants complex $\mathcal{C}$ of $\Sigma$ and 
prove that the group $\mathcal{T}$ acts transitively by automorphisms on it. 
Finally, we establish that the automorphism group of the complex 
$\mathcal{C}$ is an extension of the Thompson group 
$\mathcal{T}$ by $\mathbb{Z}/2\mathbb{Z}$.
\end{abstract}

\textbf{2000 MSC classsification: 57 N 05, 20 F 38, 57 M 07, 20 F 34.}


\section{Introduction}

Given a compact surface of genus $g$ with $b$ boundary components,
one can build various locally infinite simplicial complexes on which the
mapping class group of the surface acts by automorphisms, mainly the curve 
complex, the arc complex and the pants complex. The automorphism group of 
each one of these complexes is related to the mapping class group 
itself. Specifically, Ivanov in \cite{ivanov} and subsequently Luo in 
\cite{luo} proved that the 
automorphism group of the curve complex is the extended mapping class 
group for almost all compact surfaces, and Korkmaz extended this result to
punctured spheres and punctured tori in \cite{korkmaz}. Analogous
results were proved for the arc complex by Ivanov \cite{ivanov}
and Irmak and McCarthy \cite{irmakMccarthy}, for the pants complex 
by Margalit \cite{margalit},
for the Hatcher-Thurston complex by Irmak and Korkmaz \cite{irmakKorkmaz},
for the arc and curve complex by Korkmaz and Papadopoulos 
\cite{korkmazPapadopoulos}, and for the complex of non-separating curves
by Schmutz-Schaller \cite{schmutz} and subsequently by Irmak \cite{irmak}.  

When dealing with surfaces of infinite type, one issue is to find good 
analogues for mapping class groups and for the complexes enumerated 
above. Funar and Kapoudjan defined the notion of
asymptotic mapping class group for some surface of genus zero and 
infinite type in \cite{universalmcg}, the first one of a series
of three papers (see also \cite{braidedthompson} and \cite{combable}).  
They constructed several complexes for this type of surfaces and proved that 
the asymptotic mapping class group acts on them by automorphisms. 
These complexes can be seen as generalizations of the pants complex
to the infinite type case. It is worth to mention that the 
asymptotic mapping class groups they defined are closely related to
Thompson's group $\mathcal{T}$, which is one of the first examples of finitely
presented infinite simple groups.

In this paper, we consider the case of a planar infinite surface with non
compact boundary, which has the Thompson group $\mathcal{T}$ as
its asymptotic mapping class group. Then, we construct a locally
infinite cellular complex $\mathcal{C}$ related to the pants complex 
of the infinite surface of genus zero of \cite{universalmcg}. We prove
that the group $\mathcal{T}$ acts by automorphisms on the complex and,
finally, we establish that the automorphism group of the complex is 
$\mathcal{T} \rtimes \mathbb{Z}/2\mathbb{Z}$. We also give a 
geometric interpretation to this semi-direct product as the extended
asymptotic mapping class group, obtaining in this way a complete analogue 
to Ivanov's theorem for the surface $\Sigma$ of infinite type.

Our result can also be considered from a different viewpoint. In 1970,
Tits proved in \cite{tits} that most subgroups of automorphism groups of
trees generated by vertex stabilizers are simple.
This result was extended by Haglund and Paulin in \cite{haglundpaulin} to 
the case of automorphism groups of negatively curved polyhedral complexes 
which are (in general uncountable) non-linear and virtually simple.
On the other hand, Farley proved in \cite{farley} that the Thompson groups
$\mathcal{F}$, $\mathcal{T}$ and $\mathcal{V}$ act properly and isometrically 
on CAT(0) cubical complexes. In particular, these groups are a-T-menable and 
satisfy the Baum-Connes conjecture with arbitrary coefficients. This leads to 
the legitimate question of realizing the simple groups $\mathcal{T}$ and
$\mathcal{V}$ as automorphisms groups of suitable cell complexes. Although the 
complex which we obtain is not Gromov hyperbolic, it gives a convenient answer 
to this question for the group $\mathcal{T}$.

\subsection*{Definitions and statement of the results}

\subsubsection*{The surface $\Sigma$ and its asymptotic mapping class group}

Let $\mathbb{D}^2$ be the hyperbolic disc and suppose that its boundary 
$\partial\mathbb{D}^2$ is parametrized by the unit interval (with 
identified endpoints). Let $E$ denote the triangulation of $\mathbb{D}^2$ 
given by the family of bi-infinite geodesics representing the standard 
dyadic intervals, i.e. the family of geodesics $I_a^n$ joining the 
points $p=\frac{a}{2^n}$, $q=\frac{a+1}{2^n}$ on $\partial\mathbb{D}^2$, 
where $a,n$ are integers satisfying $0 \leq a \leq 2^n -1$. Let $T_3$ be 
the dual graph of $E$, which is an infinite (unrooted) trivalent tree.
We choose to realize the tree $T_3$ in $\mathbb{D}^2$ by using geodesic 
segments obtained from the $I_a^n$ (with $n \geq 1$ and $a$ even) by 
rotations of angles $\frac{\pi}{2^n}$, as shown in figure \ref{dual}.
We also denote by $T_3$ this realisation. 

\begin{figure} [H]
\begin{center}
\input{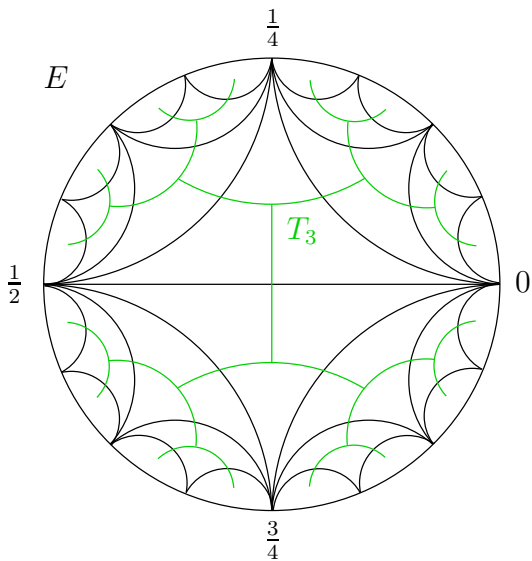}
\caption{The triangulation $E$ of $\mathbb{D}^2$ and its dual tree $T_3$.}
\label{dual}
\end{center}
\end{figure}

Now, let $\Sigma$ be a closed $\delta$-neighbourhood of $T_3$
(see figure \ref{surface}). Remark that 
the surface $\Sigma$ is planar (thus oriented), non compact, contractible and 
its boundary has infinitely many connected components. Note that we can 
choose $\delta$ small enough to avoid self-intersections of the boundary 
and such that each intersection $I_a^n \cap \Sigma$ is a single connected 
arc joining two boundary components of $\Sigma$. 

The triangulation $E$ together with the boundary components of $\Sigma$ 
define a tessellation of our surface into hexagons 
(see figure \ref{surface}). By abuse of notation 
we will denote this hexagonal tessellation also by $E$. Note that the 
boundary of each hexagon alternates sides contained in different connected 
components of $\partial\Sigma$ with sides contained in different arcs 
of the triangulation $E$. 

\begin{figure} [H]
\begin{center}
\input{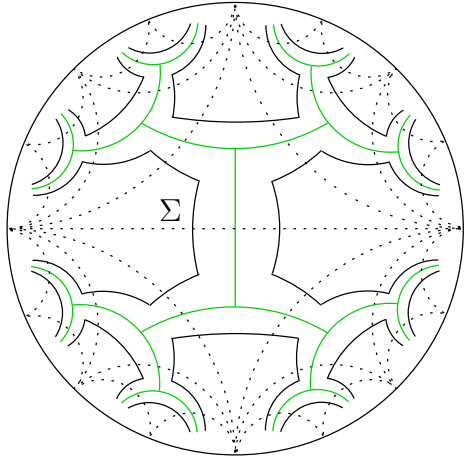}
\caption{The surface $\Sigma$ and its hexagonal tessellation.}
\label{surface}
\end{center}
\end{figure}

\begin{definition} The set of arcs $\{I_{a_1}^{n_1}, \ldots, I_{a_k}^{n_k}\}$
of $E$ is {\rm{separating}} if the union of its arcs bounds a compact 
subsurface of $\Sigma$ containing $I_{0}^{1}\cap \Sigma$.
\end{definition}

Let $Homeo(\Sigma)$ denote the group of orientation-preserving 
homeomorphisms of $\Sigma$. 

\begin{definition} An element $\varphi \in Homeo(\Sigma)$ is 
{\rm{asymptotically rigid}} if there exists a set of separating arcs 
$\{I_{a_1}^{n_1}, \ldots, I_{a_k}^{n_k}\}$ such that:
\begin{enumerate}
\item For all $1 \leq i \leq k$, there exist $I_{b_i}^{m_i}$ such that 
$\varphi(I_{a_i}^{n_i})\cap \Sigma = I_{b_i}^{m_i}\cap \Sigma$, and
\item For all $1 \leq i \leq k$, for all $j \in \mathbb{N}$, and for all
$0 \leq l < 2^j$ the arc 
$\varphi(I_{2^ja_i+l}^{n_i+j})\cap \Sigma$ is equal to the arc 
$I_{2^jb_i+l}^{m_i+j}\cap\Sigma$. 
\end{enumerate}
\end{definition} 

\begin{remark}
{\rm{The set of arcs $\{I_{b_1}^{m_1},\ldots,I_{b_k}^{m_k}\}$ from the
definition above is a set of separating arcs.}}
\end{remark}

\begin{remark}
{\rm{Let $\varphi$ be an asymptotically rigid homeomorphism.
The first condition of the definition says that there exists 
a compact subsurface $C \subset \Sigma$ which boundary is
composed by segments of $\partial \Sigma$ and arcs of $E$ such that 
$\varphi$ sends $C$ to another compact subsurface $C'$ of the same type 
(i.e. which boundary is composed by segments of $\partial \Sigma$ and arcs 
of $E$). The second
condition deals with the behaviour of $\varphi$ outside $C$. Let
$t_1, \ldots, t_k$ be the connected components of $T_3 \setminus (T_3 \cap C)$
and $t_1', \ldots, t_k'$ be the connected components of 
$T_3 \setminus (T_3 \cap C')$. Then the second condition says that each $t_i$
is 'naturally' sent to $t_i'$.}}
\end{remark}

The group of asymptotically rigid homeomorphisms of $\Sigma$ will be 
denoted by $Homeo_{a}(\Sigma)$.

\begin{definition} The {\rm{asymptotic mapping class group}} of $\Sigma$ 
is the quotient of $Homeo_{a}(\Sigma)$ by the group of isotopies of $\Sigma$. 
\end{definition}

\begin{remark} {\rm{(\cite{braidedthompson}, section 1.3)
The asymptotic mapping class group of $\Sigma$ is 
isomorphic to the Thompson group $\mathcal{T}$, which is the group of
orientation preserving piecewise linear homeomorphisms of the
circle that are differentiable except at finitely many dyadic
rational numbers, such that, on intervals of
differentiability, the derivatives are powers of 2, and which send
the set of dyadic rational numbers to itself. For
a complete introduction to Thompson's groups see \cite{CFP}.}}
\end{remark}

\subsubsection*{The complex $\mathcal{C}$}

Once the surface $\Sigma$ defined, we are going to define a 
2-dimensional cellular complex $\mathcal{C}$. The vertices 
of $\mathcal{C}$ are the isotopy classes relatively to the boundary of the
tessellations of $\Sigma$ into hexagons which differ from the 
distinguished tessellation $E$ only in a finite number of hexagons.
Note that every hexagon of a tessellation have 3 sides corresponding
to different boundary components of $\Sigma$ and 3 sides corresponding 
to arcs joining these boundary components.

Let $v$ be a vertex of $\mathcal{C}$, and let $a$ be an arc bounding
two hexagons $h_1$, $h_2$ of the tessellation $v$. Let $a'$ be an arc
contained in $h_1 \cup h_2$ joining two boundary components of $\Sigma$
such that $i([a],[a'])=1$, where $i$ denote the geometric intersection
number and $[a]$ the isotopy class of $a$ relatively to the boundary. 
Let $w$ be the vertex of $\mathcal{C}$ obtained from $v$ by replacing the arc 
$a$ by $a'$. Then, the vertices $v$ and $w$ are joined by an edge on 
$\mathcal{C}$. We say that $w$ is obtained from $v$ by \textit{flipping}
the arc $a$.

Let $v$ be a vertex of $\mathcal{C}$ and let $a$ and $b$ be two arcs,
each one of which bounding two hexagons of the tessellation $v$.
We say that $a$ and $b$ are \textit{adjacent} if there exists an hexagon on 
the tessellation $v$ having both of them on its boundary. There are two
possible cases:
\begin{enumerate}
\item The two arcs are not adjacent. 
Then, flipping first $a$ and then $b$ one obtains the same tessellation as
the result of first flipping $b$ and then $a$. 
Thus, the commutativity of flips defines a 4-cycle on the 1-skeleton of 
$\mathcal{C}$ and we fill it with a \textit{squared 2-cell}.
\item The two arcs are adjacent.
Then, flipping first $a$ and then $b$ one obtains a different tessellation
from the result of first flipping $b$ and then $a$. These two
tessellations differ from a flip.
We fill the corresponding 5-cycle with a \textit{pentagonal 2-cell}.
\end{enumerate} 

We obtain, in particular, the following analogue to the case of the
cellular complexes associated to compact surfaces:

\textbf{Proposition \ref{lemact}.} The asymptotic mapping class group of 
$\Sigma$ acts transitively on the complex $\mathcal{C}$ by automorphisms.

It is now natural to ask whether or not
the automorphism group of the cellular complex $\mathcal{C}$ is related
to the asymptotic mapping class group of the surface $\Sigma$. To answer this 
question, we construct the following exact sequence:
$$1 \longrightarrow \mathcal{T} \longrightarrow \acp
\longrightarrow \mathbb{Z}/2\mathbb{Z} \longrightarrow 1,$$
and prove that it splits. 

Thus, we obtain the following result:

\textbf{Theorem 1.} The automorphism group of the complex $\mathcal{C}$
is isomorphic to the semi-direct product 
$\mathcal{T} \rtimes \mathbb{Z}/2\mathbb{Z}$.

Furthermore, the extension is not trivial; we obtain a non trivial
action of $\mathbb{Z}/2\mathbb{Z}$ to $\text{Out}(\mathcal{T})$.
Thus, as Brin established in \cite{brin} that the group of outer automorphisms
of $\mathcal{T}$ is isomorphic to $\mathbb{Z}/2\mathbb{Z}$, we get an
isomorphism between $\text{Out}(\mathcal{T})$ and $\mathbb{Z}/2\mathbb{Z}$.

Finally, the tools used to prove the main theorem lead us to give a 
geometric interpretation of this result. It turns out that one can 
classify the elements of $\acp$ in two categories: those which preserve
(in some natural sense) the orientation of the complex, and those which
reverse the orientation. Then, it is natural to define the \textit{extended
mapping class group of $\Sigma$} as 
$\mathcal{T} \rtimes \mathbb{Z}/2\mathbb{Z}$, which makes theorem 1 a
complete analogue to Ivanov's theorem for compact surfaces.

\subsubsection*{Structure of the paper}

This paper is structured as follows. In the second section, 
we provide a modification of the surface $\Sigma$ which does 
not involve changes on the mapping class group. This will allow us 
to describe the isomorphism between the asymptotic mapping class 
group and Thompson's group $\mathcal{T}$. In the third section, 
we briefly discuss the analogy between the complex $\mathcal{C}$ and
the complexes obtained for compact surfaces and define the action
of the asymptotic mapping class group on $\mathcal{C}$. 
Finally, in the fourth section, we describe the structure of the 
automorphisms group of the complex $\mathcal{C}$, proving theorem 1.

\subsection*{Acknowledgements}
The first author acknowledges financial support from ``La Caixa'' 
and ``l'Ambassade de France en Espagne'' Post-graduated Fellowship.
Both authors would like to thank Javier Aramayona, Francis Bonahon, 
Jos\'e Burillo,
Fran\c{c}ois Dahmani, Vlad Sergiescu and Juan Souto for useful comments. 
The authors are particularly indebted to Louis Funar and Fr\'ed\'eric Mouton
for enlightening discussions and suggestions.

\section{The surface and its asymptotic mapping class group}

In this section we introduce a slightly different surface 
$\Sigma'$ which  allows us to simplify the situation. 
Both surfaces $\Sigma$ and $\Sigma'$ have the 
same asymptotic mapping class group. 

\begin{definition}
The surface $\Sigma'$ is obtained by attaching to the hyperbolic disk 
$\mathbb{D}^2$ the points of its boundary corresponding to the 
dyadic rational numbers.
\end{definition}

\begin{remark} {\rm{The triangulation $E$ of 
$\mathbb{D}^2$ is a triangulation of $\Sigma'$ having for 
vertices all the boundary components of $\Sigma'$.}} \label{triang}
\end{remark}

\begin{lemma} Let $\partial_0\Sigma$ be a connected component of the boundary
of $\Sigma$, which is a bi-infinite line on $\mathbb{D}^2$.
The adherence of $\partial_0\Sigma$ is $\partial_0\Sigma \cup \{q\}$,
where $q$ is a dyadic rational number on $\partial\mathbb{D}^2$.
Thus, there is a natural bijection between the set of dyadic 
rational numbers $\partial \Sigma'$
and the set of connected components of the boundary of $\Sigma$. 
\label{bijection}
\end{lemma}

\begin{corollary}
Let $\partial_0\Sigma$ be a connected component of the boundary
of $\Sigma$, and let $I_{a_1}^{n_1}$, $I_{a_2}^{n_2}$ be two geodesics
of $E$ intersecting $\partial_0\Sigma$. Then, $I_{a_1}^{n_1}$ and 
$I_{a_2}^{n_2}$ have a common boundary point (dyadic). \label{extreme}
\end{corollary}

\begin{corollary}
Let $\{I_{a_1}^{n_1},\ldots,I_{a_k}^{n_k}\}$ be a set of separating arcs
which are cyclically ordered. Then, they are the sides of an ideal 
hyperbolic $k$-gon.
\label{polygons}
\end{corollary}

\begin{corollary}
One can deform continuously $\Sigma$ to $\Sigma'$ by an homotopy
$H: \Sigma \times [0,1] \rightarrow \mathbb{D}^2 \cup \partial\mathbb{D}^2$ 
such that:
\begin{itemize}
\item For $t < 1$, $H(\cdot,t)$ is an homeomorphism.
\item For $t < 1$, $H(\Sigma\cap I_a^n,t) \subset I_a^n$.
\item $H(\cdot,1)$ is an homeomorphism between the interior of $\Sigma$
and $\mathbb{D}^2$.
\item Let $\partial_0 \Sigma$ be a connected component of the boundary of 
$\Sigma$ and let $q$ be the dyadic rational number on its adherence. Then,
$H(\partial_0\Sigma,1)=q$.
\end{itemize}
\label{homotopy}
\end{corollary}

One can now define the notion of asymptotically rigid homeomorphism of
$\Sigma'$, and the asymptotic mapping class group by taking 
$Homeo_a(\Sigma')$ modulo isotopies of $\Sigma'$.

\begin{proposition}
The asymptotic mapping class groups of $\Sigma$ and $\Sigma'$
are isomorphic.
\end{proposition}

Furthermore, they are isomorphic to Thompson's group $\mathcal{T}$.

\begin{proof}
Let $\varphi \in \text{Homeo}_a(\Sigma)$. We define $\varphi' \in 
\text{Homeo}_a(\Sigma')$ as follows:
$$\varphi'(x) = \left\{
\begin{array}{cl}
H\left(\varphi(y),1\right), & \text{if } x \in \mathbb{D}^2\\
H\left(\varphi\left(\partial_x\Sigma\right),1\right), & \text{if } 
x \in \partial \Sigma',
\end{array}
\right.
$$
where $y \in \Sigma$ such that $H(y,1)=x$,
and $\partial_x\Sigma$ is the connected component of $\partial \Sigma$
having $x$ on its adherence and $H$ is the homotopy from lemma \ref{homotopy}. 
Note that every isotopy of $\Sigma$ can be extended to an isotopy of 
$\mathbb{D}^2 \cup \partial \mathbb{D}^2$.
Thus, the application we defined induces an injective morphism for
the mapping class groups.

Reciprocally, let $[\varphi']$ be an element of the asymptotic  
mapping class group
of $\Sigma'$. Note that we can choose a representative 
$\varphi' \in \text{Homeo}_a(\Sigma')$ such that 
$\varphi'_{|\Sigma} \in \text{Homeo}_a(\Sigma)$. Thus, we defined
an injective morphism from the asymptotic mapping class group
of $\Sigma'$ to the asymptotic mapping class group of $\Sigma$. 

Finally, one can verify that the composition of both is the identity. 
\end{proof}

Before giving the isomorphism between $\mathcal{T}$ and the asymptotic 
mapping class group of $\Sigma$, let us give the analytical definition of Thompson's group $\mathcal{T}$:  

\begin{definition}
The {\rm{Thompson group $\mathcal{T}$}} is the set of
orientation preserving piecewise linear homeomorphisms of the
circle $S^1=[0, 1]/_{0\sim1}$ whose points of non-differentiability
are dyadic rational numbers, whose derivatives (where defined)
are powers of 2, which send the set of dyadic rational numbers to itself.
\end{definition}

Let $\varphi$ be a representative of an equivalence classe 
of the asymptotic mapping class group of $\Sigma'$. 
We claim that $\varphi$ acts as an element of the
Thompson group $\mathcal{T}$ on $\partial \Sigma'$, which is the set of
dyadic rational numbers, and that this action does not depend on the
representative. To see this, let 
$\{I_{a_1}^{n_1}, \ldots, I_{a_k}^{n_k}\}$ and
$\{I_{b_1}^{m_1}, \ldots, I_{b_k}^{m_k}\}$ be sets of separating arcs
associated to $\varphi$, with $\varphi(I_{a_i}^{n_i})=I_{b_i}^{m_i}$
for $1 \leq i \leq k$. Without loss of generality, we can
suppose that $a_1=0$, $a_k=2^{n_k}-1$ ($I_a^1$ is contained on the 
compact subsurface bounded by the arcs $I_{a}^{n}$) and that both sets 
are cyclically ordered. Let $f$ be the unique map of the circle 
satisfying the following conditions:

\begin{enumerate}
\item For all $1 \leq i \leq k, \, 
\displaystyle f\left(\frac{a_i}{2^{n_i}}\right)=
\frac{b_i}{2^{m_i}}$, 
\item For all $1 \leq i \leq k, \, 
\displaystyle f\left(\frac{a_i+1}{2^{n_i}}\right)=
\frac{b_i+1}{2^{m_i}}$, and
\item For all $\displaystyle x \in \left(\frac{a_i}{2^{n_i}},
\frac{a_i+1}{2^{n_i}}\right), \, f'(x) = 2^{n_i-m_i}$. 
\end{enumerate}

One can see that $f \in \mathcal{T}$ using the corollary \ref{polygons}
and the fact that $\varphi$ is orientation-preserving.

Reciprocally, given an element $f \in \mathcal{T}$ it can be proven 
(adaptation to $\mathcal{T}$ of lemma 2.2 from \cite{CFP}) that 
there exists a partition of the unit interval into standard dyadic
intervals $0=x_0 < x_1 < \ldots < x_k=1$,
and an integer $0 \leq j \leq k$ such that:
\begin{enumerate}
\item The subintervals of the partition 
$0=f(x_j) < f(x_{j+1}) < \ldots < f(x_k) 
= f(x_0) < \ldots < f(x_{j-1}) <f(x_j)+1 = 1$ are standard dyadic, and
\item The map $f$ is linear in every subinterval of the partition.
\end{enumerate}
Thus, we can find two ideal $k$-gons inscribed in $E$ defining
an element $\varphi$ in the asymptotic mapping class group 
of $\Sigma'$ which acts as $f$ on $\partial\Sigma'$. This ends the
proof of the following result 
(see \cite{braidedthompson}): 

\begin{proposition} 
The asymptotic mapping class group of $\Sigma$ is isomorphic
to the Thompson group $\mathcal{T}$.
\end{proposition}

For the purposes of this paper it is useful to think of the Thompson
group $\mathcal{T}$ as the group generated by two elements: 
an element $\alpha$ of order 4 and an element $\beta$ of order 3. 
The full presentation is the following
(for the details see \cite{braidedthompson} and \cite{lochak}):

$$\mathcal{T} \simeq \left<\alpha, \beta \, \vert \, \alpha^4, \beta^3, 
\left[\beta\alpha\beta,\alpha^2\beta\alpha\beta\alpha^2\right],
\left[\beta\alpha\beta,\alpha^2\beta^2
\alpha^2\beta\alpha\beta\alpha^2
\beta\alpha^2\right],
\left(\beta\alpha\right)^5\right>.$$

The two generators $\alpha, \beta$ of Thompson's group $\mathcal{T}$
can be defined in terms of polygons as follows: $\alpha$ sends 
$\{I_{0}^{2},I_{1}^{2},I_{2}^{2},I_{3}^{2}\}$ respectively to
$\{I_{1}^{2},I_{2}^{2},I_{3}^{2},I_{0}^{2}\}$ and $\beta$ sends
$\{I_{0}^{1},I_{0}^{2},I_{1}^{2}\}$ respectively to
$\{I_{0}^{2},I_{1}^{2},I_{0}^{1}\}$.  

Remark that the polygons defining an element of $\mathcal{T}$ are not unique,
although there is a minimal one which satisfies the conditions. 
Consider this minimal pair of polygons defining an element 
$f\in \mathcal{T}$. Then, if we split a standard dyadic interval of 
the source polygon in two halves and we also split the standard dyadic
interval corresponding to its image in two halves, the element defined 
by the new pair of polygons is the same. All possible pairs of
polygons defining $f$ are obtained doing this expanding operation 
finitely many times. 

\section{The asymptotic pants complex}

First of all, we would like to present the relation of the
complex $\mathcal{C}$ and the different known complexes for
the compact case. In some general sense, one
can see $\mathcal{C}$ as an analogue to the pants complex. 
For this, take another copy of the initial surface $\Sigma$, already
tessellated in hexagons by $E$, and glue the two copies together along 
their boundary. Now, the arcs of $E$ on the two copies are also glued 
and they become simple curves which decompose the doubled surface in
infinitely many pairs of pants. Hence, the vertices of the complex
$\mathcal{C}$ can be seen as decompositions in pairs of pants of 
the double surface, obtained from the initial decomposition by a finite 
number of elementary moves. An elementary move consists on changing a 
simple curve $\gamma$ by another $\gamma'$, which does not intersect the
other curves of the pants decomposition, and with minimal geometric 
intersection number $i([\gamma],[\gamma'])=2$. In our case, we 
only consider the elementary moves that can be seen on the planar surface 
$\Sigma$. For the details of this construction see \cite{universalmcg}
and \cite{braidedthompson}.  

The rest of this section is devoted to give another point of view of
the complex $\mathcal{C}$ using the definition of $\Sigma'$.
For example, we can see the 0-skeleton of 
$\mathcal{C}$ as the set of triangulations of $\Sigma'$ with vertices in 
$\partial \Sigma'$ which differ from $E$ only on a finite number of
triangles. Hence, two vertices $u,v$ of $\mathcal{C}$ 
are joined by an edge if we can obtain $v$ from $u$ by changing the 
diagonal of a quadrilateral inscribed on the triangulation $u$. 
We will that $v$ is obtained by \textit{flipping an arc of} $u$.   

\begin{remark} \label{cells} 
{\rm{(Geometric interpretation of the 2-cells) 
Each 2-cell of the 
complex can be seen, geometrically, as the object obtained from a 
triangulation $v$ of the complex $\mathcal{C}$ by flipping consecutively 
two of its arcs. If there exists a triangle in $v$ containing the two arcs 
on its boundary, then the 2-cell is pentagonal. Otherwise, the 2-cell is
squared.}}
\end{remark} 

\begin{remark} \label{adjacent} 
{\rm{Given three vertices $u,v,w$ of $\cp$ where $u,w$ are different 
and adjacent to $v$, there exists a unique 2-cell containing $u,v,w$. 
This follows from the precedent remark. }}
\end{remark}

\begin{definition}
We call the 2-dimensional cellular complex $\mathcal{C}$ the
{\rm{asymptotic pants complex}} of the surfaces $\Sigma$ and $\Sigma'$.
\end{definition}

It is worth to mention that the asymptotic pants complex is closely related
to the graph of triangulations of a convex $n$-gon (see, for example,
\cite{flipgraph}, \cite{lee} and \cite{hurtadoNoy}) and hence, it is also 
related to the rotation distance in binary trees (see \cite{STT}).
The graph of 
triangulations of a convex $n$-gon is a graph whose vertices are the 
triangulations of the $n$-gon in $n-2$ triangles, and where
two vertices are joined by an edge if there is a single flip going from
one to the other. Thus, the asymptotic pants complex can be seen as a 
2-dimensional complex associated to the triangulations graph of an
infinite-sided convex polygon. Now, one can derive the connectivity of
the complex $\mathcal{C}$ from the finite case (see \cite{hurtadoNoy} 
for a simple proof). To see this, one can think that every vertex of 
$\mathcal{C}$ has only a finite number of arcs not contained in the
triangulation $E$ and so, there exists a convex polygon having all of this
arcs on its interior. Furthermore, Lee in \cite{lee} proved
that the complex of triangulations of a convex $n$-gon is a convex polytope
of dimension $n-3$ (see also \cite{devadossRead}). Thus, the complex 
$\mathcal{C}$ is also simply connected.

\begin{proposition}
The complex $\mathcal{C}$ is connected and simply connected.
\end{proposition} 

Finally, we will define the action of $\mathcal{T}$ on $\mathcal{C}$. 
Let $v$ be a vertex of $\mathcal{C}$ and let $f$ be an element of 
$\mathcal{T}$. Consider the set of geodesics $G$ defined as follows: 
the geodesic joining the dyadic rational numbers $p,q$ is in $G$ 
if the geodesic joining the dyadic rational numbers $f^{-1}(p)$ and 
$f^{-1}(q)$ belongs to the triangulation $v$. Remark that the geodesics
on $G$ define a triangulation of $\Sigma'$; this triangulation will
be the image by $f$ of the vertex $v$ and we will denote it by $f \cdot v$.
As a consequence of the definitions of $\mathcal{T}$ 
and $\mathcal{C}$, the triangulation $f\cdot v$ corresponds to a vertex 
of $\mathcal{C}$. Note that it is possible to find an ideal polygon bounded 
by the set of separating arcs $\{I_{a_1}^{n_1}, \ldots, I_{a_k}^{n_k}\}$
which is the source polygon for the element $f$ and which contains all 
the geodesics that are on the triangulation $v$ and which are not on
the base triangulation $E$.

Let $v,w$ be two vertices joined by an edge in $\mathcal{C}$, and let
$f$ be an element of $\mathcal{T}$. It follows from the definition of the 
action of $f$ on the vertices of $\mathcal{C}$ that $f\cdot v$ and 
$f\cdot w$ are joined 
by an edge in $\mathcal{C}$. It also follows from the definition of the 
action on the vertices together with the geometric interpretation of the 
2-cells (remark \ref{cells}) that the action can be extended to the 
2-skeleton of $\mathcal{C}$. Thus, we can define a natural map 
$\Psi: \mathcal{T} \rightarrow Aut(\mathcal{C})$. 

\begin{proposition}
The asymptotic mapping class group $\mathcal{T}$ of $\Sigma'$ acts 
transitively on the asymptotic pants complex $\mathcal{C}$ by 
automorphisms. Furthermore, the map 
$\Psi: \mathcal{T} \rightarrow Aut(\mathcal{C})$ is injective. 
\label{lemact}
\end{proposition}

\begin{proof}
It follows from the paragraph preceding the proposition that $\mathcal{T}$
acts on $\mathcal{C}$ by automorphisms. We prove the transitivity by showing
that given a vertex $v$ of $\mathcal{C}$ there exists an element 
$f \in \mathcal{T}$ such that $f \cdot E = v$. Let $P$ be an ideal
polygon inscribed on $E$ containing all the edges of the triangulation 
$v$ which are not on the triangulation $E$, and having at least the
dyadics $0,\frac{1}{4},\frac{1}{2},\frac{3}{4}$ on its boundary. 
Let $a$ be an interior edge of the triangulation 
of $P$ in $v$ joining the dyadics $p$ and $q$. Suppose that $p >q$. There
exists a unique triangle in $P$ having as vertices $p, x, q$, where
$x$ is a dyadic number $p > x > q$. We associate the triangle of $E$
with vertices $\frac{1}{2}, 0, \frac{1}{4}$ to this triangle. Analogously,
we associate to the triangle with vertices $y > p > q$ (or $p > q > y$) 
of $P$ the triangle with vertices $\frac{3}{4},0,\frac{1}{2}$. 
Repeating this procedure finitely many times we construct an ideal polygon 
$Q$ triangulated as $E$,
whose sides are naturally associated to those of the polygon $P$. 
Then, the element $f \in \mathcal{T}$ defined by $f(Q)=P$ satisfies 
our claim. 

For the second assertion it is enough to prove that given an element 
$f \in \mathcal{T}$ different from the identity, there exists a vertex 
$v$ in $\mathcal{C}$ such that $f \cdot v \neq v$. Let $q$ be a dyadic 
such that $f(q) \neq q$ (it exists because $f \neq id$). Take a square 
inscribed in the triangulation $E$, having $q$ as a vertex and such that
it has empty intersection with its image by $f$. Now, let $v$ be the 
triangulation obtained from $E$ by flipping the diagonal of this square.
Then, $f\cdot v \neq v$.
\end{proof}

\begin{proposition}
The complex $\mathcal{C}$ is not Gromov hyperbolic. \label{hyp}
\end{proposition}

The proof uses the following observation:

\begin{remark}
{\rm{Every flip changes a single arc of the initial triangulation. 
Therefore, the distance between 
two vertices $v,w \in \mathcal{C}$ is bounded below by the number of arcs 
which are on the triangulation $v$ and are not in the triangulation $w$.
}} \label{distance}
\end{remark}

\begin{proof}
Let $n$ be a positive integer. We construct a geodesic triangle $u,v,w$
such that $d_{\mathcal{C}}(u,v)=d_{\mathcal{C}}(v,w)=n$, and 
$d_{\mathcal{C}}(u,w)=2n$;
and we give a point $p$ in the geodesic segment $uw$ such that 
$d_{\mathcal{C}}(p,x) \geq n$ for all $x \in uv \cup vw$.

Let $s_1,\ldots, s_{2n}$ be 2$n$ disjoint squares inscribed in the 
triangulation $E$. Let $u=E$, $v$ be the vertex obtained from $E$ by flipping 
the diagonals of the squares $s_1, \ldots, s_n$ and $w$ be the vertex obtained 
from $v$ by flipping the diagonals of the squares $s_{n+1}, \ldots, s_{2n}$.
The segments $uv$ and $vw$ consist on flipping, respectively, the diagonals 
of $s_1, \ldots, s_n$, and $s_{n+1}, \ldots, s_{2n}$, in order. The segment
$uw$ consists on flipping first the diagonal of the square $s_{2n}$, then
the diagonal of $s_{2n-1}$ until the diagonal of $s_1$, in decreasing order
(see figure \ref{triangle}).
The point $p$ is the vertex obtained from $E$ by flipping the diagonals of
$s_{n+1}, \ldots, s_{2n}$. By the remark \ref{distance} all these segments are
geodesics, and 
$d_{\mathcal{C}}(p,x)=n+\min\{d_{\mathcal{C}}(u,x),d_{\mathcal{C}}(w,x)\}$,
for all $x \in uv \cup vw$.
\end{proof}

\begin{figure} [H]
\begin{center}
\includegraphics[width=10cm]{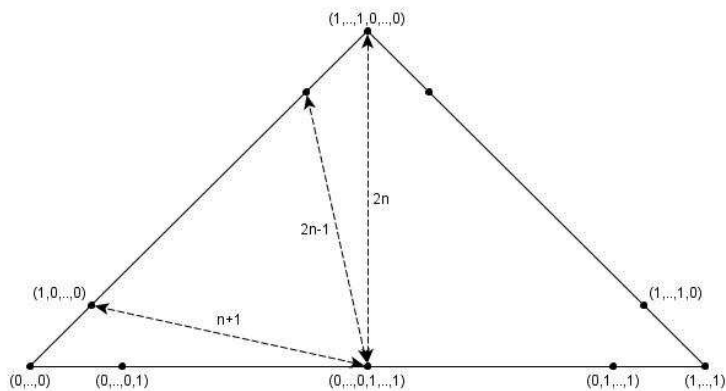}
\caption{The points correspond to triangulations obtained from $E$ by
flipping the diagonals of $s_i$ if there is a one in the $i$-th coordinate.
Discontinuous lines represent distances in $\mathcal{C}$.}
\label{triangle}
\end{center}
\end{figure}

\section{The automorphism group of the complex $\mathcal{C}$}

In this section we study the group of automorphisms $\acp$ 
of the asymptotic pants complex $\mathcal{C}$, and we prove that it is
isomorphic to the semi-direct product 
$\mathcal{T} \rtimes \mathbb{Z}/2\mathbb{Z}$.

The proof of this result has some similarities with Ivanov's proof
(see \cite{ivanov}) of the fact that the automorphism group of the 
arc complex of a compact surface is isomorphic to the extended 
mapping class group. 
In particular, he showed that the image of a maximal simplex determines
completely an automorphism. In our case, every automorphism is determined by
its image of the ball of radius one centred in some vertex $v$. 
Then, given an automorphism $\varphi$ of the complex $\mathcal{C}$, 
we construct an element $t \in \Psi(\mathcal{T}) \subset \acp$ 
such that $t(E)=\varphi(E)$. 
For this step, we need to introduce an auxiliary
complex, which is related to the link of $v$ in $\mathcal{C}$.  

\subsection{Construction of the link complex $\Ddv$}

Recall that in any triangulation $v$ of $\Sigma'$, the flips along two 
different arcs of $v$ do not commute if and only if both arcs belong to 
the boundary of some triangle of $v$ (remark \ref{cells}).
We construct a sub-complex of $\mathcal{C}$ which encodes this 
information. 

Let $\Dv$ be a graph whose set of vertices is the set of vertices of 
$\mathcal{C}$ adjacent to the vertex $v$.
Two vertices are connected by an edge if they lie in the same pentagon
in $\mathcal{C}$.

\begin{figure}[H]
\begin{center}
\input{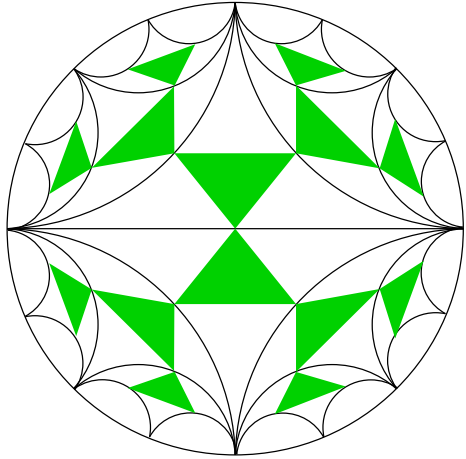}
\caption{The link of $E$.}
\label{linkE}
\end{center}
\end{figure}

Given a vertex of $\cp$ there is a geometric interpretation of $\Dv$: 
we associate to every arc $a$ of the triangulation $v$ a vertex 
which represents the vertex of $\cp$ obtained from $v$ by flipping
the arc $a$. Then, from the geometric interpretation of the 2-cells 
of $\cp$ (remark \ref{cells}) we see that two vertices of $\Dv$ 
associated to the arcs $a,b$ are joined by an edge if and only if there
is a triangle in $v$ having both $a,b$ on its boundary. Thus,
$\Dv$ can be seen as a connected union of triangles, 
each of them inscribed on a triangle of $v$, 
where the intersection between different triangles is either empty or a 
vertex, and where every vertex belong to two different triangles
(see figure \ref{linkE}). 

\begin{definition}
Let $v$ be a vertex of $\mathcal{C}$. The {\rm{link}} of $v$ is the
2-dimensional complex $\Ddv$ obtained from $\Dv$ by gluing a 2-cell on 
each triangle. Furthermore, the link lies on $\mathbb{D}^2$ and 
every triangle is naturally oriented.
\end{definition}

Remark that the link has been constructed using only the combinatorial
structure of $\mathcal{C}$ in the neighbourhood of a vertex. Note also that 
$\mathcal{C}$ is regular in the sense that all vertices have the same 
combinatorial structure on their neighbourhoods. 
Thus, one can derive the following result:

\begin{lemma} \label{lemind}
Let $v$ be a vertex of $\mathcal{C}$. Then, every automorphism 
$\phi \in \acp$ induces an isomorphism 
$\phi_{*,v}: \Ddv \rightarrow \Ddw$, where $w=\phi(v)$.
\end{lemma} 

Now, using the transitivity of the action of $\mathcal{T}$ on 
$\mathcal{C}$ one obtains:

\begin{corollary}
Let $v$ and $w$ be vertices of $\cp$. Then, their links $\Ddv$ 
and $\Ddw$ are isomorphic. 
\end{corollary}

\subsection{Extensions of link isomorphisms}

In the previous section we proved that every automorphism $\varphi$ of 
$\mathcal{C}$ induces link isomorphisms between $\Ddv$ and $\Ddw$, where
$w=\varphi(v)$. Now, given two vertices $v$ and $w$ of $\mathcal{C}$ and
given an isomorphism $i: \Ddv \rightarrow \Ddw$ between their links, it is 
natural to ask under which conditions one can find an automorphism 
$\varphi \in \acp$ such that $\varphi_{*,v}=i$. More specifically, 
which is the main obstruction to the existence of this automorphism?

Let $\varphi \in \acp$. Let $w$ be the image of $v$ by $\varphi$. Let
$(a,b,c)$ be three vertices of a triangle $\Delta$ on the link $\Ddv$, and 
suppose they are cyclically ordered according to the orientation of the
triangle. The vertices $\varphi(a),\varphi(b),\varphi(c)$ are the vertices of
a triangle $\phi(\Delta)$ in the link $\Ddw$. 
Now, there are two possible cases:
\begin{enumerate}
\item The orientation of $\varphi(\Delta)$ induces the cyclic order 
$(\varphi(a),\varphi(b),\varphi(c))$ on its vertices. Then we say that 
$\varphi$ is \textit{orientation preserving on $\Delta$}.
\item The orientation of $\varphi(\Delta)$ induces the cyclic order 
$(\varphi(a),\varphi(c),\varphi(b))$ on its vertices. Then we say that 
$\varphi$ is \textit{orientation reversing on $\Delta$}.
\end{enumerate}

\begin{definition}
Let $\varphi \in \acp$ and $v$ be a vertex of $\mathcal{C}$. We say that 
$\varphi$ is {\rm{$v$-orientation preserving (reversing)}} if it is 
orientation preserving (reversing) on every triangle of $\Ddv$. 
\end{definition}

\begin{remark}
{\rm{Every element $f \in \mathcal{T}$ induces an automorphism 
$\Psi(f) \in \acp$ (from proposition \ref{lemact}) which is 
$v$-orientation preserving for all
vertices $v$ of $\mathcal{C}$.}} \label{orientT}
\end{remark}

\begin{lemma}\label{lemauto}
Let $v$ and $w$ be vertices of $\mathcal{C}$, and $i: \Ddv \rightarrow \Ddw$
be an orientation preserving isomorphism of their links. Then, there exists 
a unique automorphism $\varphi_i \in \acp$ such that $\varphi_{*,v}=i$. 
Furthermore, $\varphi_i \in \mathcal{T}$.
\end{lemma} 

\begin{proof}
We first prove the existence by constructing $\varphi  \in \mathcal{T}$ in the 
following way: let $\{I_{c_1}^{d_1},\ldots,I_{c_l}^{d_l}\}$ a set of
separating arcs on the triangulation $w$ containing all the edges which
differ from $E$. Let $u'_1, \ldots, u'_l$ be the vertices of $\Dw$ 
corresponding to the triangulations obtained from $w$ by flipping, 
respectively, the edges $I_{c_1}^{d_1},\ldots,I_{c_l}^{d_l}$. Let
$\{I_{a_1}^{n_1},\ldots,I_{a_k}^{n_k}\}$ be separating arcs on the
triangulation $v$ containing on their interior 
$i^{-1}(u'_1),\ldots,i^{-1}(u'_l)$.
Let $u_1, \ldots, u_k$ be the vertices of $\Dv$ 
corresponding to the triangulations obtained from $v$ by flipping, 
respectively, the edges $I_{a_1}^{n_1},\ldots,I_{a_k}^{n_k}$. 
Then, $i(u_1),\ldots,i(u_k)$ are vertices corresponding to edges which 
form a set of separating arcs of $w$, 
$\{I_{b_1}^{m_1}, \ldots, I_{b_k}^{m_k}\}$. Define $\varphi$ by 
$\varphi(I_{a_j}^{n_j})=I_{b_j}^{m_j}$ for $1 \leq j \leq k$.
It is easy to see that $\varphi_{*,v}=i$. 

Uniqueness: this proof is based on the remarks \ref{cells} and \ref{adjacent}.
Let $\psi \in \acp$ such that $\psi_{*,v}=i$. Then, we prove that its
action on the 0-skeleton of $\mathcal{C}$ must coincide with $\varphi$.
\begin{itemize}
 \item $\psi(v)=w$.
 \item For every element $u$ on the unit sphere of $\mathcal{C}$ centred in  
 $v$, $\psi(u)$ must be $i(u)$.
 \item Suppose that $\psi$ is defined in the ball $\mathcal{B}_{\cp}(v,n)$
 of $\cp$ of radius $n$ centred in $v$
 for $n \geq 2$, and let $u$ be a vertex at distance $n+1$ of $v$. 
 Consider a path $p=v,u_1, \ldots, u_{n-1},u_n, u$ of length $n+1$ 
 joining $v$ and $u$. The vertices $u_{n-1},u_{n},u$ define a unique 2-cell 
 in $\cp$ (remark \ref{adjacent}), and this 2-cell contains
 at least a fourth vertex $u' \in \mathcal{B}_{\cp}(v,n)$.
 Thus, $\psi(u)$ must be the vertex of the unique 2-cell defined 
 by $\psi(u'),\psi(u_{n-1}),\psi(u_n)$ which is adjacent to $\psi(u_n)$
 and different from $\psi(u_{n-1})$.
\end{itemize}
\end{proof}

\begin{lemma}
Let $v$ and $w$ be vertices of $\cp$ and let $i: \Ddv \rightarrow \Ddw$
be an isomorphism of their links such that it is orientation reversing on
$\Delta_1=(u_0,u_1,u_2)$ and orientation preserving on 
$\Delta_2=(u_0,u_3,u_4)$. Then, there does not exist $\varphi \in \acp$ with
$\varphi_{*,v}=i$.
\label{transp}
\end{lemma}

\begin{proof}
Suppose there exists $\varphi \in \acp$ such that $\varphi_{*,v}=i$. 

\begin{figure}[H]
\begin{center}
\input{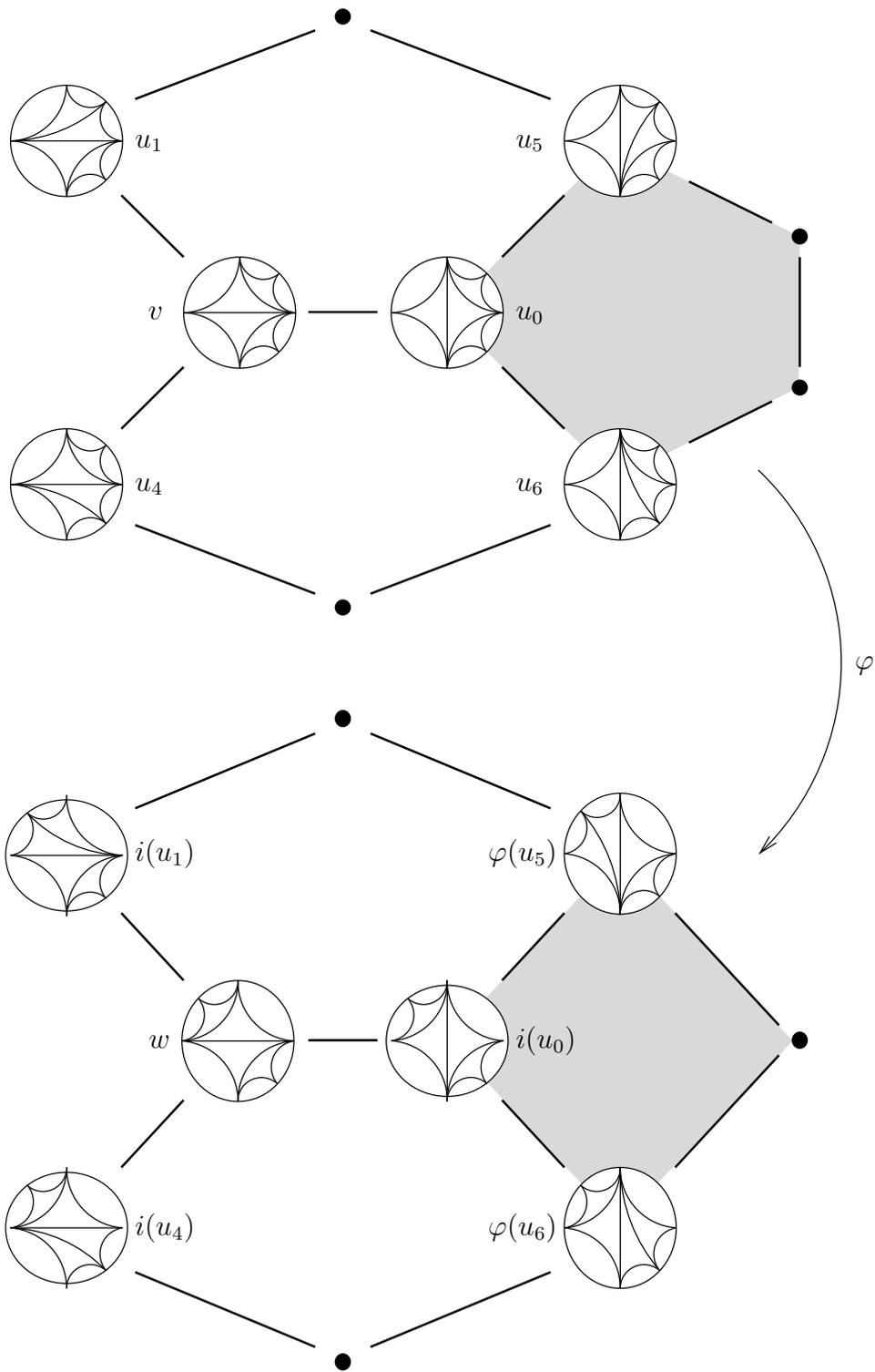}
\caption{Example of lemma \ref{transp}.}
\label{last}
\end{center}
\end{figure}

The vertices $u_1, v, u_0$ define a pentagonal 2-cell of $\mathcal{C}$;
let $u_5$ be the vertex at distance 1 of $u_0$ and different from $v$ on
this cell. Analogously, let $u_6$ be the vertex at distance 1 of $u_0$ 
different from $v$ and lying on the pentagonal 2-cell defined by the path 
$u_4,v,u_0$. The path $u_5, u_0, u_6$ defines a pentagonal 2-cell of 
$\mathcal{C}$. 

On the other hand, $i(u_1), i(v), i(u_0)$ also define a pentagonal 2-cell 
of $\mathcal{C}$. Adapting the prove of uniqueness of the lemma \ref{lemauto}, 
$\varphi(u_5)$ must be the vertex at distance 1 of $i(u_0)$ and different from 
$w$ on this pentagonal 2-cell. Furthermore, $\varphi(u_6)$ must be the vertex 
at distance 1 of $\varphi(u_0)$ different from $w$ and lying on the pentagonal 
2-cell defined by the path $i(u_4),i(v),i(u_0)$. But in this case the path 
$\varphi(u_5), \varphi(u_0), \varphi(u_6)$ defines a squared 2-cell of 
$\mathcal{C}$, which contradicts the fact that $\varphi \in \acp$. See figure
\ref{last} for a picture of the situation.
\end{proof}

\begin{remark}
{\rm{Let $v$ be a vertex of $\mathcal{C}$ and let $w$ be one of its neighbours 
in $\mathcal{C}$. Let $w,u_1,u_2$ and $w,u_3,u_4$ denote the triangles of
$\Ddv$ containing $w$. Observe that the geometric representation of the link 
$\Ddw$ of $w$ has the same vertices as $\Ddv$ and only the triangles mentioned 
above change to $w, u_1, u_4$ and $w, u_2, u_3$. See figures
\ref{linkE} and \ref{linkFlip}.}} \label{flipLinks}
\end{remark} 

\begin{figure} [H]
\begin{center}
\input{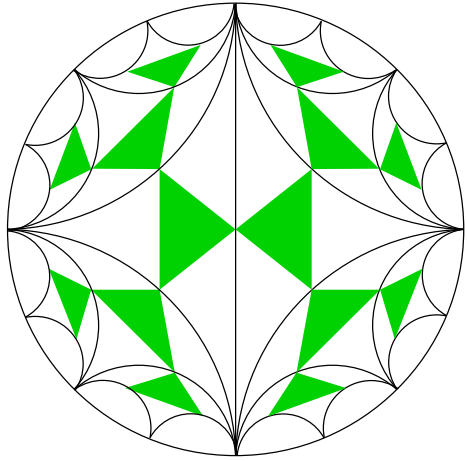}
\caption{The link of the vertex of $\mathcal{C}$ obtained from $E$
by flipping $I_0^1$.}
\label{linkFlip}
\end{center}
\end{figure}

This provides an alternative prove of lemma \ref{transp}:

\begin{proof}
Suppose there exists $\varphi \in \acp$ such that $\varphi_{*,v}=i$.
Then $\varphi_{*,u_0}$ must be an isomorphism between the link of $u_0$
and the link of $i(u_0)$. As a consequence of remark \ref{flipLinks},
$u_1$ and $u_4$ are adjacent in the link of $u_0$, but $i(u_1)$ and $i(u_4)$ 
are not adjacent in the link of $i(u_0)$ (see figure \ref{prova}), which 
contradicts the fact that $\varphi_{*,u_0}$ is an isomorphism. 
\end{proof}

\begin{figure}[H]
\begin{center}
\includegraphics[width=10cm]{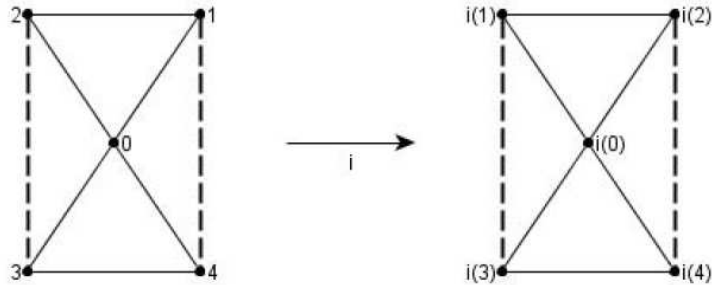}
\caption{The map $i$ in $\Delta_1$ and $\Delta_2$. In discontinuous
lines the links of $u_0$ and $i(u_0)$.}
\label{prova}
\end{center}
\end{figure}

\begin{lemma}
Let $i_R: \DdE \rightarrow \DdE$ be the orientation reversing isomorphism
obtained by the symmetry of axis $I_0^1$. Then, there exists a unique 
automorphism $\varphi_R \in \acp$ such that $\varphi_{R_{*,E}}=i_R$. 
\label{sym}
\end{lemma}

\begin{proof}
The uniqueness can be proved exactly as in lemma \ref{lemauto}. 
For the existence it suffices to send each vertex to its symmetric with
respect to the symmetry of axis $I_0^1$.
\end{proof}
 
\subsection{Proof of theorem 1}

In order to simplify the prove of theorem 1 and give a geometric 
interpretation of it, we introduce the following concept:

\begin{definition}
Let $\varphi \in \acp$. It is {\rm{orientation preserving (reversing)}} 
if it is $v$-orientation preserving (reversing) for all vertices $v$ of 
$\mathcal{C}$.
\end{definition}

\begin{remark} 
{\rm{As a consequence of lemma \ref{lemauto}, the subgroup of orientation 
preserving automorphisms of $\mathcal{C}$ is isomorphic to $\mathcal{T}$.}} 
\label{orPres}
\end{remark}

\begin{remark}
{\rm{The automorphism $\varphi_R$ from lemma \ref{sym} is orientation 
reversing.}} \label{orRev}
\end{remark}

Using remarks \ref{orPres} and \ref{orRev}, one obtains the following classification of the 
automorphisms of $\mathcal{C}$:

\begin{proposition}
Every automorphism of $\mathcal{C}$ is either orientation preserving or
orientation reversing. \label{classification}
\end{proposition}

\begin{proof}
As a consequence of lemma \ref{transp}, given an automorphism $\varphi$ 
of $\mathcal{C}$, its restriction $\varphi_{*,E}$ to the link $\DdE$
is either $E$-orientation preserving or $E$-orientation reversing. In the
first case, $\varphi \in \Psi(\mathcal{T})$ by lemma \ref{lemauto} and thus
it is orientation preserving (remark \ref{orPres}). In the second case,
consider the automorphism $\varphi'=\varphi \circ \varphi_R$, which is 
$E$-orientation preserving and thus orientation preserving. Then, 
$\varphi = \varphi' \circ \varphi_R$ (note that $\varphi_R^{-1}=\varphi_R$),
therefore it is orientation reversing (remark \ref{orRev}).
\end{proof}

Now, we prove theorem 1:

\begin{proof}
We want to construct the following exact sequence:
$$1 \longrightarrow \mathcal{T} \longrightarrow \acp
\longrightarrow \mathbb{Z}/2\mathbb{Z} \longrightarrow 1,$$
and prove that it splits.

We denote by $\Pi: \acp \rightarrow \mathbb{Z}/2\mathbb{Z}$ and
define it as follows:

$$
 \varphi  \mapsto  \left\{ \begin{array}{ll}
 0, & \text{if } \varphi \text{ is orientation preserving}\\
 1, & \text{if } \varphi \text{ is orientation reversing.}
\end{array}\right.
$$

Proposition \ref{classification} shows that the map $\Pi$ is a well
defined morphism of groups. As a consequence of lemmas \ref{lemauto} 
and \ref{sym}, the map $\Pi$ is surjective and the isomorphism 
$\mathbb{Z}/2\mathbb{Z} \simeq \left<\varphi_R\right>$ is a section of $\Pi$. 

Finally, $\text{ker}(\Pi) \simeq \mathcal{T}$ follows from lemma 
\ref{lemauto}.
\end{proof}

Therefore, we have an action of $\mathbb{Z}/2\mathbb{Z}$ into the group
of automorphisms of $\mathcal{T}$, where the generator of 
$\mathbb{Z}/2\mathbb{Z}$ acts as the automorphism $\gamma_R$ which is 
defined by $\gamma_R(t)= \varphi_R \circ t \circ \varphi_R$, for 
$t \in \mathcal{T}$. In particular, the automorphism $\gamma_R$ obtained in 
$\mathcal{T}$ is defined by $\gamma_R(\alpha)=\alpha^{-1}$ and 
$\gamma_R(\beta)=\alpha^2 \beta^{-1} \alpha^2$. Thus, if we compose $\gamma_R$
with the conjugation by $\alpha^2$ we obtain the generator of the group of 
outer automorphisms of $\mathcal{T}$ sending $\alpha$ to $\alpha^{-1}$ 
and $\beta$ to $\beta^{-1}$, which is the generator of 
$\text{Out}(\mathcal{T})$ given by Brin in \cite{brin}.

\begin{remark} 
{\rm{(Geometric interpretation of theorem 1.)
By proposition \ref{classification}, all automorphisms of $\mathcal{C}$ 
are either orientation preserving or orientation reversing. Thus, if one
defines the \textit{extended asymptotic mapping class group of $\Sigma$} as
$\mathcal{T} \rtimes \mathbb{Z}/2\mathbb{Z}$, theorem 1 turns out to be
an exact analogue to Ivanov's theorem for the planar surface of infinite
type $\Sigma$.}}
\end{remark}

\end{document}